\documentclass[12pt]{article}

\usepackage{amssymb,xcolor,graphicx,amsthm}
\usepackage{amsmath}

\newtheorem{theorem}{Theorem}[section]
\newtheorem{lemma}[theorem]{Lemma}
\newtheorem{proposition}[theorem]{Proposition}
\newtheorem{corollary}[theorem]{Corollary}

\newtheorem{example}[theorem]{Example} 

\newcommand{\ben}{\begin{enumerate}}
	\newcommand{\een}{\end{enumerate}}

\newcommand{\bt}{\begin{theorem}}
	\newcommand{\et}{\end{theorem}}
\newcommand{\bl}{\begin{lemma}}
	\newcommand{\el}{\end{lemma}}
\newcommand{\bc}{\begin{corollary}}
	\newcommand{\ec}{\end{corollary}}
\newcommand{\bp}{\begin{proposition}}
	\newcommand{\ep}{\end{proposition}}
\newcommand{\br}{\begin{remark}}
	\newcommand{\er}{\end{remark}}
	\newcommand{\be}{\begin{example}}
	\newcommand{\ee}{\end{example}}

\newtheorem{remark}{Remark}

\providecommand{\keywords}[1]
{
	\small	
	{\noindent\textit{Keywords: }} #1
}

\usepackage{amsmath}

\title{Coupling results and Markovian structures for number representations of continuous random variables}

\author{J. M\o ller}
\date{}

\begin{document}

\maketitle

\begin{abstract}
A general setting for nested subdivisions of a bounded real set into intervals defining the digits $X_1,X_2,...$ of a random variable $X$ with a probability density function $f$ is considered. Under the weak condition that $f$ is almost everywhere lower semi-continuous, a coupling between $X$ and a non-negative integer-valued random variable $N$ is established so that $X_1,...,X_N$ have an interpretation as the ``sufficient digits'', since the distribution of $R=(X_{N+1},X_{N+2},...)$ conditioned on $S=(X_1,...,X_N)$ does not depend on $f$. Adding a condition about a Markovian structure of the lengths of the intervals in the nested subdivisions, $R\,|\,S$ becomes a Markov chain of a certain order $s\ge0$. If $s=0$ then $X_{N+1},X_{N+2},...$ are IID with a known distribution. When $s>0$ and the Markov chain is uniformly geometric ergodic, a coupling is established between $(X,N)$ and a random time $M$ so that the chain after time $\max\{N,s\}+M-s$ is stationary and $M$ follows a simple known distribution. The results are related to several examples of number representations generated by a dynamical system, including base-$q$ expansions, generalized L\"uroth series, $\beta$-expansions, and continued fraction representations. The importance of the results and some suggestions and open problems for future research are discussed.
\end{abstract}

\keywords{Approximation;
beta-expansion; 
continued fraction;  
dynamical system;  
L\"{u}roth series; 
Markov chain;
nested partitions;
pseudo golden ratio; 
read once algorithm;
sufficient digits.}

\section{Introduction}\label{s:intro}

Let $X$ be a continuous random variable with 
digit expansion $X=0.X_1X_2...$ in the usual decimal numeral system (base-10) or another numeral system as described below. This paper shows that under very mild conditions, a non-negative integer-valued random variable $N$ exists  so that $S=(X_1,...,X_N)$ are the ``sufficient digits'' in the sense that the distribution of $R=(X_{N+1},X_{N+2},...)$ conditioned on $S=(X_1,...,X_N)$ does not depend on the distribution of $X$. Under a further condition ensuring a Markovian structure of the numeral system, which is satisfied for base-10 and many {\color{black}other} numeral systems, it is also shown that $R\,|\,S$ is a Markov chain of a certain order $s\ge0$. When $s=0$, which e.g.\ is satisfied for base-10, $X_{N+1},X_{N+2},...$ are shown to be IID with a known distribution. When $s>0$ and the Markov chain is uniformly geometric ergodic, 
a coupling is established between $(X,N)$ and a random time $M$ so that the chain after time $\max\{N,s\}+M-s$ is stationary and $M$ follows a simple known distribution. Before further presenting these results and the outline of this paper,
the
 background material on numeral systems presented in Section~\ref{s:1.1} 
is needed. The importance of the results are wide ranging as discussed in
Section~\ref{s:conclusion}.  


\subsection{Nested subdivisions}\label{s:1.1}

The general setting of nested subdivisions described in this section is considered throughout the text. Specific  constructions and examples are given in Section~\ref{s:examples}.

Let $\mathcal X$ be a finite or countable set, and
denote $\mathbb R,\mathbb N,$ and $\mathbb N_0$ the sets of real numbers, positive integers, and non-negative integers, respectively, and $\mathcal X^{\mathbb N}$ the set of all sequences $(x_1,x_2,...)$ with elements in  $\mathcal X$. 

Suppose $I=I_\emptyset\subset\mathbb R$ is a Borel set of Lebesgue measure $\ell_\emptyset=|I|\in(0,\infty)$ and for every $n\in\mathbb N$ there is a subdivision (at level $n$)
\begin{equation}\label{e:partition}
I=\bigcup_{(x_1,...,x_n)\in \mathcal X^n}I_{x_1,...,x_n}\bigcup  D_n
\end{equation}
such {\color{black}that} the following properties are satisfied. The $I_{x_1,...,x_n}$'s and $D_n$  are pairwise disjoint sets.
Each $I_k$ with  $k\in\mathcal X$  is an open non-empty interval, 
and for $n>1$, 
 $I_{x_1,...,x_n}$ is an open and possibly empty interval 
 such that $I_{x_1,...,x_n}\subseteq I_{x_1,...,x_{n-1}}$. 
 Each $D_n$ is a subset of the endpoints of the non-empty $I_{x_1,...,x_n}$'s, {\color{black}and so $D_n\subseteq D_{n+1}$.} 

Some comments, further notation, and one more assumption are in order. 

{\color{black} For illustrative purposes, in the well-known ``base-10 case'', $\mathcal X=\{0,1,...,9\}$, $I=[0,1)$, and for every $n\in\mathbb N$ and $(x_1,...,x_n)\in\{0,1,...,9\}^n$, $I_{x_1,...,x_n}$ has left and right endpoints $\sum_{i=1}^n x_i10^{-i}$ and $10^{-n}+\sum_{i=1}^n x_i10^{-i}$, respectively, and $D_n$ is the set of all such left endpoints. This example can be much generalized as discussed later in Examples~\ref{ex:a}--\ref{ex:2}. Continuous fractions (Example~\ref{ex:contfrac} below) provides an  example where $\mathcal X$ is infinite countable and non-equidistant subdivisions are used. Furthermore, considering non-equidistant subdivisions will be of importance for statistical models as discussed in Section~\ref{s:f1}.}

When $n\ge2$, $I_{x_1,...,x_n}$
is allowed to be empty, 
since restrictions on the digits may be needed in some cases (e.g.\ Example~\ref{ex:2} below).

 Let $D=\cup_{n=1}^\infty
D_n$ be the set of all endpoints. Since each $D_n$ is finite or countable, $D$ is finite or countable, and so $|D|=0$. {\color{black} }

It is only for convenience that each interval $I_{x_1,...,x_n}$ is assumed to be open. 
The assumption  is of no importance for the probabilistic results in this paper, since $|D|=0$ and all considered probability distributions on $I$ will be assumed to be absolutely continuous. Therefore, ignoring Lebesgue nullsets, the partitions given by \eqref{e:partition} for $n=1,2,...$ are said to be nested subdivisions of $I$. 

Let $J=I\setminus D$. Then $|I|=|J|$ {\color{black} and
\begin{equation}\label{e:Jeq}
J=
\bigcap_{n=1}^\infty\bigcup_{(x_1,...,x_n)\in\mathcal X^n}I_{x_1,...,x_n}\setminus D.
\end{equation}
}

For every $(x_1,x_2,...)\in\mathcal X^{\mathbb N}$ and $n\in\mathbb N$, let $\ell_{x_1,...,x_n}=|I_{x_1,...,x_n}|$ be the length of $I_{x_1,...,x_n}$ and assume 
\begin{equation}\label{e:convergence}
\ell_{x_1,...,x_n}\to0 \quad\mbox{as }n\to\infty.
\end{equation}
Denote $a_{x_1,...,x_n}$ the left endpoint of $I_{x_1,...,x_n}$.  For
 $x\in I_{x_1,...,x_n}$, 
the error $e_{n}(x)$ and the relative error $u_{n}(x)$ of the approximation 
$x\approx a_{x_1,...,x_n}$ as compared to $\ell_{x_1,...,x_n}$ are defined by 
\[e_{n}(x)=x-a_{x_1,...,x_n},\quad u_{n}(x)=e_{n}(x)/\ell_{x_1,...,x_n}. \]
By \eqref{e:convergence}, $e_{n}(x)\to0$ as $n\to\infty$. If $x\in J$ then $0<u_n(x)<1$.

Define the digits (or coefficients or symbols) of every $x\in J$ as the unique sequence $(x_1,x_2,...)\in\mathcal X^{\mathbb N}$ such that
$x\in I_{x_1,...,x_n}$ for $n=1,2,...$. 
The assumption \eqref{e:convergence} means that $x=\lim_{n\to\infty}a_{x_1,...,x_n}$ is determined  by its digits. Express this 
one-to-one correspondence between $x\in J$ and its digits by writing
\begin{equation}\label{e:natural}
x=.x_1x_2...
\quad\mbox{for $x\in J$}.
\end{equation} 
{\color{black} This notation agrees with the usual digit representations used for the example of the base-10 case.}

For $n\in\mathbb N_0$ and $x\in J$, call $x^{[n]}=.x_{n+1}x_{n+2}...$ and $x_{>n}=(x_{n+1},x_{n+2},...)$
the $n$-th (scaled) remainder of $x$ and the $n$-th tail of $x$, respectively, where  
$x^{[0]}=x$. Let $J_{> 0}=\{x_{>0}\,|\,x\in J\}$. Then, for any fixed $n\in\mathbb N_0$, 
there is a one-to-one correspondence between $x^{[n]}\in J$ and $x_{> n}\in J_{>0}$. 

\subsection{Objective and main results}\label{s:1.3}

\subsubsection{{\color{black}Overview}}

Let $X$ be a random variable with values in $J$ and a probability density function (PDF) $f$ concentrated on $J$. 
Denote its digits by $X=0.X_1X_2...$, cf.\ \eqref{e:natural}. Then $(X_1,X_2,...)$ is a stochastic process with state space $\mathcal X$. For $n\in\mathbb N$,
writing $A_n=a_{X_1,...,X_n}$ and $L_n=\ell_{X_1,...,X_n}$, {\color{black} then by \eqref{e:Jeq} and \eqref{e:convergence},
$A_n<X<A_n+L_n$ and $L_n\to0$ as $n\to\infty$.} 
The aim of this paper is to show how coupling methods can be used for establishing a number of interesting results for the distribution of 
\begin{equation}\label{e:aaa}
X^{[n]}=.{X_{n+1}X_{n+2}...}\quad\mbox{or equivalently}\quad X_{> n}=(X_{n+1},X_{n+2},...),
\end{equation}
the distributions of
\begin{equation}\label{e:bbb}
E_n=e_n(X)=X-A_n\quad\mbox{and}\quad U_n=u_n(X)=E_n/L_n,
\end{equation}
and the distributions of further stochastic variables/processes obtained when $n\in\mathbb N_0$ in \eqref{e:aaa} and \eqref{e:bbb} is replaced by a random time obtained from one of the
 various coupling constructions studied in detail in Sections~\ref{s:coupling-const} and \ref{s:Markov}. 
These distributions are of interest when considering number representations of $X$ as in Section~\ref{ex:exs} below as well as many other examples 
(see \cite{Karma,Dajani2,HerbstEtAl,HerbstEtAl2,Fritz} and the references therein). {\color{black} The main results in this paper extend those included in \cite{HerbstEtAl} (in which only base-$q$ expansions with $q\in\{2,3,...\}$ is considered), however the proofs are shorter and different. In particular, under weak conditions the existence of a random variable $N\in\mathbb N_0$ will be established so that the distribution of the residual process 
 \begin{equation}\label{e:star3}
 R=(R_1,R_2,...)=(X_{N+1},X_{N+2},...)=X_{>N}
 \end{equation} 
 conditioned on 
 \begin{equation}\label{e:star1}
S=(X_1,...,X_N)
\end{equation} 
does not depend on the distribution of $X$. Surprisingly, this distribution depends only on ratios of the lengths of intervals in the nested subdivisions and often a Markovian structure appears as discussed in Section~\ref{s:details} below.
Moreover, the
results have important applications for statistical methodology and metric number theory as discussed in Section~\ref{s:conclusion}.}

\subsubsection{{\color{black}Details}}\label{s:details}

The weak condition  \eqref{e:LSC} below is used when establishing any result in Sections~\ref{s:coupling-const} and \ref{s:Markov}. 
Recall that lower semi-continuity (LSC) of $f$ at a point $x\in I$ means $\lim\inf_{{x_n\to x}}f(x_n)$ $\ge f(x)$
for any sequence of points $x_n$ converging to $x$. Commonly used PDFs are LSC almost everywhere. 
Consider the condition that
\begin{equation}\label{e:LSC}
\mbox{$f$ is LSC on a Borel set $H\subseteq J$ where $|J\setminus H|=0$}
\end{equation} 
(hence $|I\setminus H|=0$). This is indeed a very mild condition -- for an example where it is violated, see Remark~1 in \cite{HerbstEtAl3}. {\color{black}The set $H$ is just introduced for generality since a PDF is only unique up to a Lebesgue nullset. The choice of $H$ is commented in Remark~1 below.}

Assuming \eqref{e:LSC}, there is a useful coupling between $X$ and a non-negative integer-valued random variable $N$. The coupling construction is detailed in Theorem~\ref{thm1}, which is summarized below using the following notation. 
If $n=0$, interpret  $(x_1,...,x_n)$ as the ``empty sequence of digits'', which is denoted $\emptyset$, and set $\mathcal X^0=\{\emptyset\}$ (this notation is similar to the use of $\emptyset$ for the empty point configuration in connection to finite point processes \cite{JM}). 
{\color{black} Then $S$ defined by \eqref{e:star1}}
becomes a discrete random variable with state space $\cup_{n=0}^\infty\mathcal X^n$, the set of all finite sequences of digits including the case of no digits. 
 For $\ell>0$,
denote $\mathrm{Unif}(0,\ell)$ and $\mathrm{Unif}(I)$ the uniform distribution on $(0,\ell)$ and $I$, respectively. 
For mathematical convenience, introduce two random variables $E_0$ and $U_0$ such that conditioned on $N=0$, $S\perp\! \! \!\perp E_0\sim\mathrm{Unif}(0,\ell_\emptyset)$ and  
$S\perp\! \! \!\perp U_0\sim\mathrm{Unif}(0,1)$ where $\perp\! \! \!\perp$ means independence. Define
\begin{equation}\label{e:star2}
 E=E_N,\quad U=U_N=E/L\quad\mbox{where }L=
 \begin{cases}\ell_N &\text{ if }N>0\\
 \ell_{\emptyset}&\text{ if }N=0.
 \end{cases}
\end{equation} 
 Then
$(X,N)$ and $(S,U)$ are 
in a one-to-one correspondence.

Among other things Theorem~\ref{thm1} states that $S\perp\! \! \!\perp U\sim\mathrm{Unif}(0,1)$; conditioned on $L$, one has $S\perp\! \! \!\perp E\sim\mathrm{Unif}(0,L)$; and the conditional distribution of
$R\,|\,S$ does not depend on $f$ but is expressible in terms of ratios between the lengths of the intervals in the nested subdivisions. These results extend those in \cite{HerbstEtAl} which only considered the special case of base-$q$ expansions (with
$q>1$ an integer, cf.\ Example~\ref{ex:a} below; see also \cite{part1,part2}). 

The results in Section~\ref{s:coupling-const} are expanded in Section~\ref{s:Markov}, assuming
the lengths of the intervals in 
the nested subdivisions of $I$ exhibit the following structure.
Set $0/0=0$ and for a given $s\in\mathbb N_0$, assume that
for any integer $n\ge s+2$ and every $(x_1,...,x_n)\in\mathcal X^n$, the ratio $\ell_{x_1,...,x_{n}}/\ell_{x_1,...,x_{n-1}}$ depends only on $(x_{n-s},...,x_{n})$, which is written as
\begin{equation}\label{e:MC2}
\ell_{x_1,...,x_{n}}/\ell_{x_1,...,x_{n-1}}= p_{x_{n-s},...,x_{n}}.
\end{equation}
Then the lengths of the intervals are said to exhibit a Markovian structure of order $s$ because $R\,|\,S$ turns out to be a Markov chain of order $s$ as explained below.

The special case $s=0$ is treated in Section~\ref{s:coupling-const2}, where it is shown that $X_1,X_2,...$ are independent identically distributed (IID) with a known distribution. Suppose $s>0$. 
Define 
\begin{equation}\label{e:omega}
\Omega=\{(x_1,...,x_{s})\in \mathcal X^{s}\,|\, \ell_{x_1,...,x_{s}}>0\}
\end{equation}
and a matrix $P$ with entries indexed by $((x_1,...,x_s),(y_1,...,y_s))\in\Omega\times\Omega$ such that
\begin{equation}\label{e:Ps}
P_{(x_1,...,x_s),(y_1,...,y_s)}=\begin{cases}
p_{x_1,...,x_{s},y_s} & \text{if }y_i=x_{i+1}\text{ for }1\le i<s\\
0 & \text{otherwise.}
\end{cases}
\end{equation}
For $i\in\mathbb N$, let $X^{(s)}_i=(X_i,X_{i+1},...,X_{i+s-1})$ and $R^{(s)}_i= 
 X^{(s)}_{N+i}$, and for $m\in\mathbb N_0$, let
\begin{align}\label{e:def-not1}
X^{(s)}&=(X^{(s)}_1,X^{(s)}_2,...), 
\quad X^{(s)}_{>m}=(X^{(s)}_{m+1},X^{(s)}_{m+1},...),\\ 
R^{(s)}&=(R^{(s)}_1,R^{(s)}_2,...) 
,\quad \,\,R^{(s)}_{>m}=(R^{(s)}_{m+1},R^{(s)}_{m+2},...) 
,\label{e:def-not2}
\end{align}
so $R^{(s)}_{\color{black}{>m}}=X^{(s)}_{>N+m}$. It is shown in Section~\ref{s:Markov1} that $R^{(s)}\,|\,S$ is a Markov chain with transition matrix $P$. In particular, if the chain $R^{(s)}\,|\,S$ is uniformly geometric ergodic, a coupling of $(X,N)$ with a random time $M$ is established in Section~\ref{s:Markov2} such that $X^{(s)}_{>\max\{N,s\}+M-s}$ is stationary and $M$ is independent of the chain and follows a simple known distribution. This coupling construction is based on the read once algorithm for perfect simulation
developed by Wilson \cite{wilson}. 

\subsection{Outline}\label{s:out}

Section~\ref{s:examples} discusses various procedures of constructing nested subdivisions and presents examples used later on. The coupling construction between $X$ and $N$ and related results are given in Section~\ref{s:coupling-const}. Under various conditions, Section~\ref{s:Markov} shows that $R^{(s)}\,|\,S$ is a Markov chain,  
specifies the coupling between $(X,N)$ and $M$, and establishes related results.
Section~\ref{s:conclusion} concludes with a brief discussion of the importance of the results 
and some open problems.

\section{Procedures for nested subdivisions}\label{s:examples}

Assuming the first subdivision of $I$ into the intervals $I_k=(a_k,a_k+\ell_k)$ with $k\in\mathcal X$ is given, it is 
explained in Section~\ref{s:2.1} how probability distributions and in Section~\ref{s:1.2} how dynamical systems may be used to generate nested subdivisions of $I$. Various illustrative examples are presented in Section~\ref{ex:exs}, including generalized L\"uroth series, $\beta$-expansions, and continued fraction representations.

\subsection{Markovian constructions}\label{s:2.1}

Given the subdivision of $I$ into the intervals $I_k=(a_k,a_k+\ell_k)$ with $k\in\mathcal X$, an integer $s\ge0$, and a probability mass function (PMF)
  on $\mathcal X^m$ with $m=s+1$, 
the following procedure is in accordance with the Markovian structure in \eqref{e:MC2}. 

Consider first the case $s=0$ and define a PMF by $p_k=\ell_k/\ell_\emptyset$ for all $ k\in\mathcal X$. In accordance with \eqref{e:MC2}, for each $(x_1,...,x_n)\in\mathcal X^{n}$ with $n\in\mathbb N$, let
$\ell_{x_1,...,x_{n}}=p_{x_1}\cdots p_{x_{n}}\ell_\emptyset$. Recursively, for $n=1,2,...$, assuming $a_{x_1,...,x_n}$ has been specified, let the subdivision of $I_{x_1,...,x_n}$ 
into the 
intervals $I_{x_1,...,x_n,k}$ 
with $k\in\mathcal X$ be such that 
$a_{x_1,...,x_n,j}<a_{x_1,...,x_n,k}$ if $a_j<a_k$. 
  This ordering is such that the corresponding $\ell_{x_1,...,x_n,k}$'s are proportional to the $\ell_k$'s (the lengths of the intervals in the first subdivision of $I$). Of course another ordering could have been used as well.

Consider next the case $s{\color{black}\ge}1$. Let $p_k$ be as above and for $(x_1,...,x_{n+1})\in\mathcal X^{n+1}$ with $1\le n\le s$, suppose $p_{x_1,...,x_n,x_{n+1}}$ is a PMF with respect to $x_{n+1}\in\mathcal X$ such that $p_{x_1}\cdots p_{x_1,...,x_m}$ becomes a PMF with respect to $(x_1,...,x_{m})\in\mathcal X^m$. 
%
For 
$(x_1,...,x_{n+1})\in\mathcal X^{n+1}$ with $n\in\mathbb N$, let 
$$\ell_{x_1,...,x_{n+1}}=\ell_{x_1,...,x_n}\times\begin{cases}
p_{x_1,...,x_n,x_{n+1}}&\text{if }n\le s\\
p_{x_{n-s+1},...,x_n,x_{n+1}}&\text{if }n> s.
\end{cases}$$ 
If $\ell_{x_1,...,x_n}>0$, specify the left endpoints of the intervals in the subdivision of $I_{x_1,...,x_n}$ as above, i.e., 
$a_{x_1,...,x_n,j}<a_{x_1,...,x_n,k}$ if $a_j<a_k$. Thus \eqref{e:MC2} is satisfied and the nested subdivisions of $I$ are specified such that the ordering of the intervals at the
 first level is reused for the intervals in the subdivision of any interval at a later level.

\subsection{Dynamical systems}\label{s:1.2}

A common way to construct nested partitions is by using a dynamical system as follows (see also \cite{Karma,Dajani2,Fritz}). 
Let $T:I\mapsto I$ be a  function and $I=\cup_{k\in\mathcal X}\tilde I_k$ a countable partition such that for  each $k\in\mathcal X$, $\tilde I_k$ is a bounded interval of length $\ell_k>0$ and the restriction  of $T$ to $\tilde I_k$ is a strictly monotone function which is denoted $T_k$. For $n=1,2,...$, let $T^n(x)=T(T^{n-1}(x))$ be the state of the dynamical system at time $n$ when it is started at $T^0(x)=x$, and let the $n$-th digit of $x$ be $x_n=k$ if $T^n(x)\in \tilde I_k$.  For every $(x_1,...,x_n)\in\mathcal X^n$ with $n\in\mathbb N$, define 
$$\tilde I_{x_1,...,x_n}=\{x\in I\,|\,T^{i-1}(x)\in\tilde I_{x_i}, i=1,...,n\}$$
which by induction is seen to be a bounded interval, and let
\begin{equation}\label{e:dyn-system}
\mbox{$I_{x_1,...,x_n}$ agree with $\tilde I_{x_1,...,x_n}$ if the endpoints of $\tilde I_{x_1,...,x_n}$ are excluded.}
\end{equation}
As mentioned in Section~\ref{s:1.1}, it is of no importance for the results in this paper but only for convenience that the endpoints  are excluded. In the dynamical system literature, $\tilde I_{x_1,...,x_n}$ is called a cylinder.

Henceforth, ``a setting for a dynamical system'' refers to a situation where a dynamical system as above is given such that \eqref{e:natural} holds and each $I_{x_1,...,x_n}$ is specified by \eqref{e:dyn-system}. By induction, if $x=.x_1x_2...\in J$, then 
$x^{[n]}=T^n(x)\in J$ for $n=0,1,...$. Thus, since $X\in J$,  
\[X^{[n]}=T^n(X)\in J\quad\mbox{for }n=0,1,...\]
Moreover, if each $T_k$ is a diffeomorphism on $I_k$ except for a Lebesgue nullset, then
for every $n\in\mathbb N_0$, the distribution of $X^{[n]}$ is absolutely continuous and has a PDF on $J$ which for almost all $x=.x_1x_2...\in J$ is given by
\begin{equation}\label{e:fTn}
f^{[n]}(x)=\sum_{(k_1,...,k_n)\in\mathcal X^n:\,\ell_{k_1,...,k_n}>0}
{f\left( .k_1...k_nx_1x_2... 
\right)}\left|\left(T_{k_1}^{-1}\circ\cdots\circ T_{k_n}^{-1}\right)'(x)\right|
\end{equation} 
This follows from the transformation theorem of densities by considering each case of $X\in I_{k_1,...,k_n}$, 
and 
since $T_{k_1}^{-1}\circ\cdots\circ T_{k_n}^{-1}(x)=.k_1...k_nx_1x_2...$.

By definition, the process $X_{>0}=(X_1,X_2,...)$ is stationary if and only if $X_{>0}$ and $X_{>1}$ are identically distributed, which in the present setting can be reformulated: Stationarity is equivalent to that 
$X=0.X_1X_2...\sim f$ implies $X^{[1]}=0.X_2X_3....\sim f$, and for short $f$ is then said to be shift invariant.
In the metric number theory literature (e.g.\ \cite{Karma,Dajani2,Fritz}), 
a distribution $\mu$ on $(I,\mathcal B)$ is said to be $T$-invariant if $\mu(A)=\mu(T^{-1}(A))$ for all $A\in\mathcal B$. Then $\mu$ being absolutely continuous and $T$-invariant corresponds to that $f$ is shift invariant.

\subsection{Examples}\label{ex:exs}

Using the approach in Section~\ref{s:2.1}, it may not always be obvious to see if the nested subdivisions could also be constructed using a setting for a dynamical system as in Section~\ref{s:1.2} -- an exception is the case $s=0$ which is treated in Example~\ref{ex:a} below. On the other hand, 
many examples of important number representations are generated by a dynamical system, and they are not all of the `Markovian type'. A few examples are summarised in this section and used for illustrative purposes later on in the text. Further examples can be found in e.g.\ \cite{Karma,Dajani2,Fritz}.

In the following Examples~\ref{ex:a}--\ref{ex:contfrac}, first a dynamical system is specified and second it is understood that nested partitions are constructed as in Section~\ref{s:1.2}. 

\be\label{ex:a} 
Suppose that $I$ agrees with $[0,1]$ perhaps except at 0 or 1, 
 and a dynamical system is given 
such that \eqref{e:natural} holds and for every $k\in\mathcal X$, the map $T_k$ restricted to $I_k$ is a linear bijection into $(0,1)$ with slope $t_k /\ell_k$ where $t_k\in\{\pm1\}$. This
is the setting for a generalized L\"{u}roth series (GLS), cf.\ \cite{Barrionuevo}. Since the signs $t_k$ can vary, it is convenient that it is not needed in the present paper to specify whether $I$ and $\tilde I_k$ include or not each of their endpoints. Note that $T_{x_n}\circ\ldots\circ T_{x_1}$ restricted to $I_{x_1,...,x_n}$ is a linear bijection $I_{x_1,...,x_n}\mapsto(0,1)$ with slope $(t_{x_1}/\ell_{x_1})\cdots (t_{x_n}/\ell_{x_n})$. Thus,
\[\ell_{x_1,...,x_n}=\ell_{x_1}\cdots\ell_{x_n}\quad\mbox{for all }(x_1,...,x_n)\in\mathcal X^n,\]
and so the Markovian structure in \eqref{e:MC2} is satisfied when $s=0$. Moreover, the left endpoints of the intervals in the subdivision of $I_{x_1,...,x_n}$ are ordered in the same way as in Section~\ref{s:2.1} if 
 $t_{x_1}\cdots t_{x_n}=1$, and in the reverse order if  $t_{x_1}\cdots t_{x_n}=-1$ (i.e., $a_{x_1,...,x_n,j}>a_{x_1,...,x_n,k}$ if $a_j<a_k$). Consequently,  
if $t_k=1$ for all $k\in\mathcal X$, 
 the construction from Section~\ref{s:2.1} results in the same nested partitions, and the construction can be modified so that it works in any case: For the subdivision of $I_{x_1,...,x_n}$, use the same procedure as in Section~\ref{s:2.1} except if $t_{x_1}\cdots t_{x_n}=-1$, in which case the left endpoints should appear in the reverse order. 
 
 One important and well-known case occurs when $q\ge2$ is an integer, $I=[0,1)$, $\mathcal X=\{0,1,...,q-1\}$, and for each $k\in\mathcal X$, $\tilde I_k=[k/q,(k+1)/q)$ and $T_k(x)=qx-k$. Then every $x\in [0,1)$ has a base-$q$ expansion
\begin{equation}\label{e:sumx}
x=\sum_{i=1}^\infty x_i q^{-i}\quad\mbox{with }x_i=\lfloor q T^{i-1}(x)\rfloor,\ i=1,2,...,
\end{equation} 
and $J$ is the set of those $x$ where the sum is not finite and there exists no $n\in\mathbb N$ so that $x_n=x_{n+1}=....=q-1$. Here, $\lfloor a\rfloor$ is the integer part of $a\ge0$.  
Moreover, $a_{x_1,...,x_n}=\sum_{i=1}^n x_i q^{-i}$, $\ell_{x_1,...,x_n}=q^{-n}$, and $T^n(x)=\sum_{i=1}^\infty x_{i+n}q^{-i}$.

Another special case is an `original' L\"{u}roth series 
where $I=(0,1)$ and $\mathcal X=\{2,3,...\}$, and where 
L\"{u}roth \cite{Luroth} showed that any $x\in(0,1)$ has a series
\begin{equation}\label{e:sumxx}
x=\frac{1}{x_1}+\frac{1}{x_1(x_1-1)x_2}+...+\frac{1}{x_1(x_1-1)\cdots x_{n-1}(x_{n-1}-1)x_n}+...
\end{equation}
with $(x_1,x_2,...)\in\mathcal X^{\mathbb N}$ (sometimes in the literature, e.g.\ in \cite{Dajani2}, 0 is included into $I$ but again this is of no importance for the present paper).
Defining
 for each $k\in\mathcal X$,
$\tilde I_k=[1/k,1/(k-1))$ and $T_k(x)=(k-1)kx-k+1$, it turns out that each digit $x_n$ is uniquely given by $x_n=k$ if $1/k<T^n(x)<1/(k-1)$ (expressions of $T^n(x)$ and technical details for the case $T^n(x)=1/k$ are given in \cite{Dajani2}; fortunately this case can be ignored in the present paper).
It follows that $J$ is the set of those $x$ where the sum in \eqref{e:sumxx} is not finite.
\ee

\be\label{ex:2} 
Let
$\beta>1$, 
$I=[0,1)$, and $T(x)=\beta x-\lfloor \beta x\rfloor$. If $\beta=q\in\{2,3,...\}$, then define $\mathcal X$ as in the case of the base-$q$ expansion above. Else let  $\mathcal X=\{0,1,...,\lfloor\beta\rfloor\}$, $\tilde I_k=[k/\beta,(k+1)/\beta)$ if $0\le k<\lfloor\beta\rfloor$, and $\tilde I_{\lfloor\beta\rfloor}=[\lfloor\beta\rfloor/\beta,1)$. 
 R\'{e}nyi \cite{renyi} showed that every $x\in[0,1)$ has a so-called $\beta$-expansion
\begin{equation}\label{e:beta-exp}
x=\sum_{i=1}^\infty x_i\beta^{-i}\quad\mbox{with }x_i=\lfloor \beta T^{i-1}(x)\rfloor,\ i=1,2,...
\end{equation}
which of course agrees with \eqref{e:sumx} when $\beta\in\{2,3,...\}$.  
The convergence of the series in \eqref{e:beta-exp} implies that \eqref{e:natural} holds.

Denote the set of all admissible sequences by
 $$\mathcal A_\beta=\{(\lfloor \beta T^{0}(x)\rfloor,\lfloor \beta T^{1}(x)\rfloor,...)\,|\,0\le x<1\},$$ 
 referring to \cite{Dajani2,parry} for the detailed description of $\mathcal A_\beta$. 
Clearly, $J_{>0}\subset \mathcal A_\beta$. For instance, if $\beta=q\in\{2,3,...\}$, then $A_\beta=\mathcal X^{\mathbb N}$ and $\mathcal A_\beta\setminus J_{>0}$ is the set of base-$q$ fractions, i.e., numbers $x$ where the sum in \eqref{e:sumx} is finite.


Another tractable case is the following.
 Let $s\in\mathbb N$, $m=s+1$, and 
$\beta=\beta_m>1$ be the pseudo golden mean of order $m$, that is, the positive root of $z^m-z^{m-1}-...-z-1=0$. 
Then $\beta_2<\beta_3<...<\lim_{n\to\infty}\beta_n=2$, where $\beta_2=(1+\sqrt5)/2$ is the golden mean.
So $\mathcal X=\{0,1\}$. Using \eqref{e:beta-exp} and induction for $n=1,2,...$, it is straightforwardly verified that for any $(x_1,...,x_n)\in\mathcal X^n$, 
$\ell_{x_1,...,x_n}$ is equal to
\begin{align*}
\beta^{-n}\times
\begin{cases}
0 &\text{if }n>s,\ x_i=...=x_{i+s}=1,\ 1\le i\le n-s\\
1 &\text{else if }x_n=0\\
\beta^i-\sum_{j=0}^{i-1}\beta^j &\text{else if }x_{n-i}=0,\ x_{n-i+1}=...=x_n=1,\ 1\le i\le s,\ i<n\\
\beta^n-\sum_{j=0}^{n-1}\beta^j &\text{else if }x_1=...=x_n=1,\ s\ge n.
\end{cases}
\end{align*}
Thus,  the admissible sequences of 0's and 1's are those with at most $s$ consecutive 1's
 (in accordance with \cite{parry}). It is easily seen that 
\eqref{e:MC2} holds with $p_{x_1,...,x_s,1}=1-p_{x_1,...,x_s,0}$ and
\begin{equation}\label{e:Pex}
p_{x_1,...,x_{s},0}=\begin{cases}
{1}/{\beta}         &\text{if }x_s=0   \\
{1}/{(\beta^{i+1}-\sum_{j=1}^i\beta^j)} &\text{if }x_{s-i}=0,\ 
x_{s-i+1}=...=x_s=1,\ 1 \le i<s\\
1&\text{if } x_1=\ldots=x_s=1.
\end{cases}
\end{equation}
Finally, $\Omega=\mathcal X^s$.
   \ee

\be\label{ex:contfrac} 
Recall the following about continued fractions \cite{Kin}.
Let $I=(0,1]$, $\mathcal X=\mathbb N$, $\tilde I_k=(1/(k+1),1/k]$ for $k\in\mathcal X$, and $T(x)=1/x-\lfloor 1/x\rfloor$ be the Gauss map \cite{gauss}. 
A rational number in $I$ has a (regular) finite continued fraction representation 
$$[0;x_1,...,x_n]=
\frac{1}{x_1+\frac{1}{x_2+\frac{1}{\ddots\frac{1}{x_n}}}}\in(0,1]
$$
and each irrational number $x\in(0,1)$ has a unique infinite continued fraction representation
\[x=[0;x_1,x_2,...]= \lim_{n\to\infty}[0;x_1,...,x_n],\]
with $x_1,x_2,...\in\mathbb N$. 
So $J$ consists of all irrational numbers between 0 and 1, and
\[\tilde I_{x_1,...,x_n}=\begin{cases}
\left[\,[0;x_1,...,x_n],[0;x_1,...,x_n+1]\,\right) &\text{for $n=2,4,6,...$}\\
\left(\,[0;x_1,...,x_n+1],[0;x_1,...,x_n]\,\right] &\text{for $n=1,3,5,...$}
\end{cases}
\]
For any $x=[0;x_1,x_2,...]\in J$,
the continued fraction remainder is given by $T^n(x)=[0;x_{n+1},x_{n+2},...]\in J$.
Defining $q_n=x_nq_{n-1}+q_{n-2}$ for $n\in\mathbb N$ with $q_{-1}=0$ and $q_0=1$, then $q_n$ is a strictly increasing function of $n$ and
$\ell_{x_1,...,x_n}^{-1}={q_n(q_n+q_{n-1})}$,
so
$\ell_{x_1,...,x_{n+1}}/\ell_{x_1,...,x_n}={q_n(q_n+q_{n-1})}/[{q_{n+1}(q_{n+1}+q_{n})}]$ depends on every digit $x_1,...,x_{n+1}$. Thus, \eqref{e:MC2} cannot be satisfied for any $s\in\mathbb N_0$.
\ee

\section{The sufficient digits}\label{s:coupling-const}

The following Theorem~\ref{thm1} establishes 
 the coupling between $X$ and a non-negative integer-valued random variable $N$ under the weak assumption that $f$ is almost everywhere LSC, cf.\ \eqref{e:LSC}. As noticed after Theorem~\ref{thm1}, $S=(X_1,...,X_N)$ becomes a sufficient statistic (if $N=0$ then $S$ is interpreted as the empty set). Therefore, $X_1,...,X_N$ are called the sufficient digits (though it is of course not meant that the approximation $X\approx A_n$ does not improve when $n>N$).  
 
\subsection{Coupling construction between $X$ and $N$} 

Some notation is needed. Assume \eqref{e:LSC} and recall the definition of $H$. For $n\in\mathbb N$ and $(x_1,...,x_n)\in\mathcal X^n$, let $c_\emptyset=i_\emptyset=\inf_H f$ and  
\[
c_{x_1,...,x_n}=i_{x_1,...,x_n}-i_{x_1,...,x_{n-1}}\]
{where }
\[i_{x_1,...,x_n}=\begin{cases}
\inf_{H\cap I_{x_1,...,x_n}}f&\text{if }H\cap I_{x_1,...,x_n}\not=\emptyset\\
0&\text{if }H\cap I_{x_1,...,x_n}=\emptyset.
\end{cases}
\]
If $n=0$, interpret $(x_1,...,x_n)\in\mathcal X^n$ as $\emptyset\in\{\emptyset\}$,
 $c_{x_1,...,x_n}$ as $c_\emptyset$, and $\ell_{x_1,...,x_n}$ as $\ell_\emptyset$.
Let $\nu$ be
the product measure on $H\times\mathbb N_0$ given by the product of Lebesgue measure on $H$ and counting measure on $\mathbb N_0$.

\begin{theorem}\label{thm1}
Assuming 
\eqref{e:LSC} and recalling the notation in \eqref{e:bbb} and \eqref{e:star1}--\eqref{e:star3}, the following items (A)--(D) hold. 
\begin{enumerate}
\item[(A)] There is a
coupling of $X$ with a non-negative integer-valued random variable $N$ such that the distribution of $(X,N)$ is absolutely continuous with respect to $\nu$, with density for any $(x,n)\in H\times \mathbb N_0$ given by
\begin{equation}\label{e:res1}
f(x,n)=c_{x_1,...,x_n} \quad\mbox{if }x=.x_1x_2...
\end{equation}
\item[(B)] 
For any $n\in\mathbb N_0$, $(x_1,...,x_n)\in\mathcal X^n$, 
and $x\in H$ with $f(x)>0$, 
\begin{align}\label{e:PXN}
\mathrm P(S=(x_1,...,x_n))&=c_{x_1,...,x_n}\ell_{x_1,...,x_n},
\\
\mathrm P(N\le n\,|\,X=x)&=i_{x_1,...,x_n}/f(x),\label{e:PN}
\\
\label{e:PNn}
\mathrm P(N\le n)&=\sum_{(x_1,...,x_n)\in\mathcal X^n}i_{x_1,...,x_n}\ell_{x_1,...,x_n}.
\end{align}
\item[(C)] One has
\begin{equation}\label{e:SU}
S\perp\! \! \!\perp U\sim\mathrm{Unif}(0,1)
\end{equation}
and  $R$ conditioned on $S=(x_1,...,x_n)$ has finite dimensional distributions
\begin{equation}\label{e:fdd}
\mathrm P(R_1=x_{n+1},...,R_m=x_{n+m}\,|\,S=(x_1,...,x_n))=\ell_{x_1,...,x_{n+m}}/\ell_{x_1,...,x_n}
\end{equation}
when $(x_1,...,x_{n+m})\in\mathcal X^{n+m}$ with $m\in\mathbb N$, $n\in\mathbb N_0$, and $\mathrm P(S=(x_1,...,x_n))>0$, i.e., provided $c_{x_1,...,x_n}\ell_{x_1,...,x_n}>0$.
\item[(D)] 
   Conditioned on $L$, one has $S\perp\! \! \!\perp E\sim\mathrm{Unif}(0,L)$.
\end{enumerate}
\end{theorem}

\begin{proof}
The theorem is an extension of results in \cite{HerbstEtAl} for the special case of base-$q$ expansions (with
$q>1$ an integer, cf.\ Example~\ref{ex:a}). A general,
different, and shorter
proof is given below.
 
Let $x=.x_1x_2...\in H$.
Then, 
\[f(x)=\lim_{n\to\infty}i_{x_1,...,x_n}=\lim_{n\to\infty}\left(i_\emptyset+\sum_{i=1}^n c_{x_1,...,x_i}\right)=\sum_{n=0}^\infty c_{x_1,...,x_n}\]
where the first equality follows since $f$ is LSC at $x\in H$ and $x$ is in the interior of $I_{x_1,...,x_n}$, and the second equality follows from the definition of $i_{x_1,...,x_n}$. Clearly, $c_{x_1,...,x_n}\ge0$ and 
\[\int \sum_{n=0}^\infty c_{x_1,...,x_n}\, \mathrm dx=\int f(x)\,\mathrm dx=1,\]
so \eqref{e:res1} specifies a well-defined distribution on $H\times\mathbb N_0$.
Thereby the existence of the coupling with density \eqref{e:res1} follows, and so (A) is verified. 

From \eqref{e:res1} follows immediately \eqref{e:PXN}.
A straightforward calculation gives first \eqref{e:PN} using \eqref{e:res1} and second \eqref{e:PNn} using 
\eqref{e:PN}. This proves (B).

Since $(X,N)$ and $(S,U)$ are 
in a one-to-one correspondence, combining \eqref{e:res1} and \eqref{e:PXN} shows that for $u\in(0,1)$ and $\mathrm P(S=(x_1,...,x_n))>0$,
\begin{align*}
\mathrm P(U\le u\,&|\,S=(x_1,...,x_n))=\frac{\mathrm P(a_{x_1,...,x_n}<X\le a_{x_1,...,x_n}+ \ell_{x_1,...,x_n}u,N=n)}{\mathrm P(S=(x_1,...,x_n))}\\
&=\frac{\int_{a_{x_1,...,x_n}}^{a_{x_1,...,x_n}+u\ell_{x_1,...,x_n}} c_{x_1,...,x_n}\,\mathrm dx}{c_{x_1,...,x_n}\ell_{x_1,...,x_n}}=u.
\end{align*}
Thus, $U|S\sim\mathrm{Unif}(0,1)$ and hence \eqref{e:SU} follows. Furthermore, \eqref{e:res1} and
\eqref{e:PXN} give
\begin{align*}
&\mathrm P(R_1=x_{n+1},...,R_m=x_{n+m}\,|\,S=(x_1,...,x_n))\\
=\,&\frac{\mathrm P(X_1=x_1,...,X_n=x_n,X_{n+1}=x_{n+1},...,X_{n+m}=x_{n+m},N=n)}{\mathrm P(S=(x_1,...,x_n))}\\
=\,&\frac{\int_{I_{x_1,...,x_{n+m}}}c_{x_1,...,x_n\,\mathrm dx}}{c_{x_1,...,x_n}\ell_{x_1,...,x_n}}=\ell_{x_1,...,x_{n+m}}/\ell_{x_1,...,x_n}
\end{align*}
whereby \eqref{e:fdd} is proven. Finally, (D) is immediately obtained from \eqref{e:SU}, which completes the proof. 
\end{proof}

\begin{remark}
{\it 
Various comments on Theorem~\ref{thm1} are in order. 

The joint distribution of $(X,N)$ is specified by $X\sim f$.
Since $N$ is not a function of $X$, it may be treated as a latent variable.  Suppose 
a statistical model class  for the PDF $f$ is given (so that $f$ is LSC almost everywhere on $J$). Since $(X,N)$ is in a one-to-one correspondence with $(S,U)$, where $S\perp\! \! \!\perp U\sim\mathrm{Unif}(0,1)$,
it follows that $S$ is a sufficient statistic and $U$ is an ancillary statistic when imagining $N$ is observable, cf.\ \cite{B-N}.
Some suggestions on how to construct models for $S$ appear in Section~\ref{s:conclusion}.
 

The distribution of $N$ depends on the choice of $H$, that is, the version of the PDF $f$. In brief, to make $N$ small, pick $H$ in \eqref{e:LSC} as small as possible, cf.\ \eqref{e:PN}. 

{\color{black} By \eqref{e:PNn}, $\mathrm P(N\le n)$ is large if $f$ is approximately constant on the intervals $I_{x_1,...,x_n}$ where $\ell_{x_1,...,x_n}$ is not small.}

It is remarkable that the conditional distribution of $R$ given $S$, 
 depends not on $f$ but only on the ratios between the lengths of the intervals, cf.\ \eqref{e:fdd}. Therefore, for any $n\in\mathbb N_0$, conditioned on $N\le n$, the distribution of $X_{>n}$ (or equivalently $X^{[n]}$) depends only on these ratios. This property is exploited further in Section~\ref{s:Markov}. 
}
\end{remark}

\subsection{Corollaries}\label{s:coupling-const2}
 
 From \eqref{e:PNn} one easily gets the following corollary.

\begin{corollary}\label{c:first-cor}
Assume 
\eqref{e:LSC}. Then, for any fixed $n\in\mathbb N_0$,
\begin{equation}\label{e:PNn1}
P(N\le n)=1
\end{equation}
if and only if for every $(x_1,...,x_n)\in\mathcal X^n$, 
\begin{equation}\label{e:step}
\mbox{ $f(x)$ is constant for almost all $x\in I_{x_1,...,x_n}$.}
\end{equation}
In particular, $N=0$ almost surely if and only if $X\sim\mathrm{Unif}(I)$.
\end{corollary}

\be\label{ex:contfrac2}
For continued fraction representations, where $T$ is the Gauss map given in Example~\ref{ex:contfrac}, 
it is well-known 
that
\begin{equation}\label{e:fG}
f_{\mathrm{G}}(x)=\frac1{\ln 2}\frac1{1+x},\quad 0<x\le 1,
\end{equation}
is the PDF of the unique $T$-invariant distribution,  
see e.g.\ \cite{Karma}. By Corollary~\ref{c:first-cor}, $N$ is unbounded when $X\sim f_{\mathrm{G}}$.
\ee

The simplest case of \eqref{e:MC2} happens when $s=0$. Defining the PMF
$$\pi_{\mathrm{inv}}(k)=\ell_k/\ell_\emptyset\quad\mbox{for } k\in\mathcal X,$$
then 
\eqref{e:MC2} with $s=0$ is equivalent to
\begin{equation}\label{e:c1}
\ell_{x_1,...,x_n}/\ell_\emptyset=\pi_{\mathrm{inv}}(x_1)\cdots\pi_{\mathrm{inv}}(x_n) 
\quad \text{whenever } n\ge2 \text{ and } (x_1,...,x_n)\in\mathcal X^n.
\end{equation} 

\begin{corollary}\label{c:Luroth}
Assuming 
\eqref{e:LSC} and \eqref{e:c1}, the following items (A)--(D) hold.
\begin{enumerate}
\item[(A)] The following statements (I)--(III) are equivalent.
\begin{enumerate}
\item[(I)] $X_1,X_2,...$ are IID with PMF $\pi_{\mathrm{inv}}$.
\item[(II)] $X\sim \mathrm{Unif}(I)$.
\item[(III)] $N=0$ almost surely. 
\end{enumerate}
\item[(B)] $S\perp\! \! \!\perp R$ and
$R_1,R_2,...$ are IID with PMF $\pi_{\mathrm{inv}}$.
\item[(C)] $S\perp\! \! \!\perp   X^{[N]}\sim\mathrm{Unif}(I)$.
\item[(D)] For every $n\in\mathbb N_0$, conditioned on $N\le n$, one has $X^{[n]}\sim\mathrm{Unif}(I)$ and 
$X_{n+1},X_{n+2},...$ are IID with PMF $\pi_{\mathrm{inv}}$.
\end{enumerate}
\end{corollary}

\begin{proof} Let $(x_1,...,x_{n+m})\in\mathcal X^{n+m}$, $m\in\mathbb N$, and $n\in\mathbb N_0$. For $\mathrm P(S=(x_1,...,x_n))>0$ or equivalently $c_{x_1,...,x_n}\ell_{x_1,...,x_n}>0$, cf.\ \eqref{e:PXN}, it follows from \eqref{e:fdd} and \eqref{e:c1} that
\[\mathrm P(R_1=x_{n+1},...,R_m=x_{n+m}\,|\,S=(x_1,...,x_n))=\pi_{\mathrm{inv}}(x_{\color{black}{n+1}})\cdots\pi_{\mathrm{inv}}(x_{\color{black}{n+m}}).
\]
Thus (B) is verified. It is straightforwardly seen that (A) holds because of the following equivalences: From (B) and (III) follow (I); 
(II) implies (III); and (I) and (II) are equivalent. Finally, since $X^{[N]}=.R_1R_2...$, both (C) and (D) follow from (A) and (B).  
\end{proof}

For instance, \eqref{e:c1} is satisfied in the setting for a generalized L\"uroth series (GLS), 
including e.g.\ base-$q$ expansions, cf.\ Example~\ref{ex:a}. 
{\color{black} Conversely, given a PMF $\pi_{\mathrm{inv}}$ on $\mathcal X$, \eqref{e:c1} may be used for constructing nested subdivisions of $I$ where  each non-empty interval $I_{x_1,...,x_n}$ (interpreted as $I$ if $n=0$) is divided into intervals $I_{x_1,...,x_n,k}$ with proportions $\pi_{\mathrm{inv}}(k)$ for $k\in\mathcal X$.}
For the GLS setting, 
 the fact that (II) implies (I) in Corollary~\ref{c:Luroth} also follows from Lemma~1 in \cite{Barrionuevo}.

Let $Q^{[n]}$ be the probability distribution of $X^{[n]}$, $\mu$ a probability measure on the measure space $(I,\mathcal B)$ where 
 $\mathcal B$ is the class of Borel sets included in $I$, and $\|Q^{[n]}-\mu\|_{\mathrm{TV}}=\sup_{A\in\mathcal B}|Q^{[n]}(A)-\mu(A)|$ the total variation distance of the signed measure $Q^{[n]}-\mu$. It is well-known that if $\mu$ has PDF $h$ and $X^{[n]}$ has PDF $f^{[n]}$, then 
 \[\|Q^{[n]}-\mu\|_{\mathrm{TV}}=\frac12\int_I |f^{[n]}(x)-h(x)|\,\mathrm dx.\]
 For instance,  $f^{[n]}$ may be given by  \eqref{e:fTn}, and in the following corollary, $h=1/\ell_\emptyset$.


\begin{corollary}\label{c:coupling1}
Assume \eqref{e:LSC} and \eqref{e:c1}, and let $\mu_{\mathrm{inv}}=\mathrm{Unif}(I)$. Then there is a coupling between $(X,N)$ and a random variable $Y\sim\mu_{\mathrm{inv}}$ such that $X^{[n]}=Y^{[n]}\sim\mu_{\mathrm{inv}}$ when $N\le n$. {\color{black}In particular,}
\begin{equation}\label{e:coupin}
\|Q^{[n]}-\mu_{\mathrm{inv}}\|_{\mathrm{TV}}\le \mathrm P(N>n)\to0\quad\mbox{as }n\to\infty.
\end{equation}
\end{corollary}

\begin{proof} 
 In consistence with (A) and (B) in Corollary~\ref{c:Luroth}, let $S\perp\! \! \!\perp Y=.Y_{\color{black}1}Y_{\color{black}2}...\sim\mu_{\mathrm{inv}}$ and $R_1=Y_{N+1},R_2=Y_{N+2},...$. Conditioned on $N\le n$, one has $X^{[n]}=Y^{[n]}\sim\mu_{\mathrm{inv}}$, cf.\ (D) in Corollary~\ref{c:Luroth}. {\color{black} Since $X^{[n]}\sim Q^{[n]}$ and $Y^{[n]}\sim\mu_{\mathrm{inv}}$, the coupling inequality (see e.g.\ \cite{thorisson}) gives $\|Q^{[n]}-\mu_{\mathrm{inv}}\|_{\mathrm{TV}}\le \mathrm P(X^{[n]}\not=Y^{[n]})$, and so the inequality in \eqref{e:coupin} follows.}
 Since $N$ is finite, the convergence to 0 in \eqref{e:coupin} is obvious.
\end{proof}

\section{Markov chains}\label{s:Markov}

This section treats the case of the Markovian structure in \eqref{e:MC2}.
The special case $s=0$ has been treated in Section~\ref{s:coupling-const2}. Throughout this section, let $s\in\mathbb N$ be given, let $X\sim f$ where $f$ satisfies \eqref{e:LSC}, and let $S,N,R$ be as in Theorem~\ref{thm1}. {\color{black} Of course it is natural to pick the value of $s$ as small as possible such that \eqref{e:MC2} holds.}

\subsection{When the residual process is Markov}\label{s:Markov1}

Recall the notation in \eqref{e:omega}--\eqref{e:def-not2}. Observe that 
\eqref{e:MC2} with $n=s+1$ gives
\begin{equation}\label{e:p-def}
p_{x_1,...,x_s,x_{s+1}}= {\ell_{x_1,...,x_{s+1}}}/{\ell_{x_1,...,x_{s}}}\quad\mbox{for $(x_1,...,x_s),(x_2,...,x_{s+1})\in\Omega$}.
\end{equation}
{\color{black} It follows by induction that \eqref{e:MC2} is equivalent to
\begin{equation}\label{e:Markovlenghts}
\ell_{x_1,...,x_n}=\ell_{x_1,...,x_s}\,p_{x_1,...,x_{s+1}}\cdots p_{x_{n-s},...,x_n}
\end{equation}
whenever $n\in\{s+2,s+3,...\}$ and $(x_1,...,x_s),...,(x_{n-s+1},...,x_n)\in\Omega$. Conversely, given both $s\in\mathbb N$, nested subdivisions of $I$ at levels $n=1,...,s$, and transition probabilities  $p_{x_1,...,x_{s+1}}\ge0$ for $(x_1,...,x_s)\in\Omega$ and $x_{s+1}\in\mathcal X$ such that $\sum_{k\in \mathcal X}p_{x_1,...,x_{s},k}=1$ and $p_{x_1,...,x_{s},k}=0$ if $(x_2,...,x_s,k)\not\in\Omega$, then \eqref{e:Markovlenghts} may be used for constructing nested subdivisions of $I$ where each non-empty interval $I_{x_1,...,x_n}$ with $n\ge s$ is divided into intervals $I_{x_1,...,x_n,k}$ with proportions $p_{x_{n-s+1},...,x_{n},k}$ for $k\in\mathcal X$. 
}

\begin{theorem}\label{t:MC}
Assume 
\eqref{e:LSC} and \eqref{e:MC2}.
Consider any $n\in\mathbb N_0$ and $(x_1,...,x_n)\in\mathcal X^n$ with $\mathrm P(S=(x_1,...,x_n))>0$, that is, 
$c_{x_1,...,x_n}\ell_{x_1,...,x_n}>0$, cf.\ \eqref{e:PXN}. Conditioned on $S=(x_1,...,x_n)$, $R^{(s)}$ is a time-homogeneous Markov chain with transition matrix $P$ and  initial distribution
\begin{align}\label{e:init}
\mathrm P(R_1=x_{n+1},...,&R_s=x_{n+s}\,|\,S=(x_1,...,x_n))\nonumber\\
&=\begin{cases}
\prod_{j=n+1-s}^n p_{x_j,...,x_{j+s}} & \text{if }n\ge s\\
(\ell_{x_1,...,x_{s}}/\ell_{x_1,...,x_n} )\prod_{j=1}^n p_{x_j,...,x_{j+s}} & \text{if }n< s
\end{cases}
\end{align}
where $\prod_{j=1}^0\cdots$ is interpreted as 1.
\end{theorem}

\begin{proof} 
By \eqref{e:Ps} and \eqref{e:p-def}, $P$ is a transition matrix, since its entries are non-negative and each row sums to 1. {\color{black}For short write $\ell_{i:j}=\ell_{x_i,...,x_j}$, $p_{i:j}=p_{x_i,...,x_j}$, $c_{i:j}=c_{x_i,...,x_j}$, and $P_{i:j}=\prod_{k=i}^j p_{k:k+s}$ when $i\le j$. Then,
for any 
$m\in\mathbb N_0$ and $(x_{1},...,x_{n+m+s})\in\mathcal X^{n+m+s}$ with
$c_{1:n}\ell_{1:n}>0$,
\begin{align}
&\mathrm P(R_1=x_{n+1},..., R_{m+s}=x_{n+m+s}\,|\,S=(x_1,...,x_n))\nonumber\\
=&\,{\ell_{1:n+m+s}}/{\ell_{1:n}}\label{e:mm}\\
=&\,
\begin{cases}
P_{n+1-s:m+n} & \text{if }n\ge s\\
(\ell_{1:s}/\ell_{1:n})P_{1:n+m}  &\text{if }n< s
\end{cases}\nonumber\\
=&\,
\begin{cases}
P_{n+1-s:n}P_{n+1:n+m}&\text{if }n\ge s\\
({\ell_{1:{s}}}/{\ell_{1:n}})
P_{1:n}P_{n+1:n+m}
&\text{if }n< s
\end{cases}\label{e:shit}
\end{align}
where} \eqref{e:mm} is obtained from \eqref{e:fdd}, $0/0=0$, and \eqref{e:shit} follows from \eqref{e:MC2} and \eqref{e:p-def}. 
From \eqref{e:shit} with $m=0$ follows \eqref{e:init}. If $m>0$, then
\begin{align*}
&\mathrm P(R_{m+s}=x_{n+m+s}\,|\,S=(x_1,...,x_n),R_1=x_{n+1},..., R_{m+s-1}=x_{n+m+s-1})\\
=&\,\frac{\mathrm P(R_1=x_{n+1},..., R_{m+s}=x_{n+m+s}\,|\,S=(x_1,...,x_n))}{\mathrm P(R_1=x_{n+1},..., R_{m+s-1}=x_{n+m+s-1}\,|\,S=(x_1,...,x_n))} \\
=&\, {\color{black}p_{{n+m}:{n+m+s}}} 
\end{align*}
where \eqref{e:shit} is used in the last equality.
So $R^{(s)}\,|\,S=(x_1,...,x_n)$ is a time-homogeneous Markov chain with transition matrix $P$. 
\end{proof}



The following will be needed in Corollary~\ref{c:inv-PDF} below.  


To any 
PMF $\pi=(\pi_{x_1,...,x_s})_{(x_1,...,x_s)\in\Omega}$ corresponds a PDF $f_\pi$ on $H$ which is defined for every $x=.x_1x_2...\in H$ by
\begin{equation}\label{e:finv}
f_{\pi}(x)= \begin{cases}
{\pi_{x_1,...,x_s}}/{\ell_{x_1,...,x_s}}&\text{if }(x_1,...,x_s)\in\Omega\\
0&\text{otherwise}.
\end{cases}
\end{equation}
Clearly, $f_\pi$ is LSC on $H$. Conversely, from any PDF $f$ on $H$ which is constant almost everywhere on each interval $I_{x_1,...,x_s}$ with $(x_1,...,x_s)\in \Omega$, one obtains a PMF $\pi$ on $\Omega$ such that \eqref{e:finv} is satisfied with $f=f_\pi$. 

Assume \eqref{e:MC2} and $X\sim f_\pi$.  
Then $f_\pi$ is shift invariant if and only if
$\pi$ is invariant with respect to $P$ (or for short $P$-invariant)  
which means that for every $(x_1,...,x_s)\in\Omega$,
\begin{equation}\label{e:inv}
\sum_{k:\,(k,x_1,....,x_{s-1})\in\Omega}\pi_{k,x_1,....,x_{s-1}}p_{k,x_1,...,x_s}=\pi_{x_1,...,x_s}.
\end{equation} 
For $s=1$, this is interpreted as $\sum_{k\in\mathcal X}\pi_k p_{k,x_1}=\pi_{x_1}$. 
 
\begin{corollary}\label{c:inv-PDF}
Assume \eqref{e:MC2}
 and $X\sim f_\pi$ where $\pi$ is $P$-invariant. 
 Then $N\le s$ almost surely and $X^{(s)}$
 is a stationary Markov chain with initial distribution $(X_1,...,X_s)\sim\pi$  and transition matrix $P$.
Conditioned on $(X_1,...,X_s)$, one has $N\perp\! \! \!\perp X_{>s}$.
\end{corollary}
 
\begin{proof} 
By \eqref{e:finv}, from $X\sim f_{\pi}$ follows $(X_1,...,X_s)\sim\pi$, and  $N\le s$ almost surely, cf.\ Corollary~\ref{c:first-cor}. By Theorem~\ref{t:MC}, conditioned on both $N$ and $(X_1,...,X_s)$, or equivalently, conditioned on $S$ and on those (if any) $R_i$ with $N< i\le s-N$, $X^{(s)}$
is a Markov chain with transition matrix $P$ and
the distribution of $X_{>s}$
depends only on $(X_1,...,X_s)$ (and not on $N$). Therefore, conditioned on $(X_1,...,X_s)$, one has $N\perp\! \! \!\perp X_{>s}$,  
and since $(X_1,...,X_s)\sim\pi$, the distribution of the chain $X^{(s)}=\{X^{(s)}_i\}_{i=1,2,...}$ is invariant under left shifts. The proof is completed.
\end{proof} 
 
 It is well-known that there exists a unique $P$-invariant PMF if and only if $P$ is irreducible. 
 The notation $\pi_{\mathrm{inv}}$ will be used to indicate that this is the unique $P$-invariant PMF. Letting $f_{\mathrm{inv}}=f_{\pi_{\mathrm{inv}}}$, one gets: 

\begin{corollary}\label{c:stat-inv-PDF}
Assume \eqref{e:MC2} and  
$P$ is irreducible. 
Then $f_{\mathrm{inv}}$ is the unique shift invariant PDF (up to equivalence of PDFs).
\end{corollary}

The condition in Corollary~\ref{c:stat-inv-PDF} that $P$ is irreducible translates by \eqref{e:MC2} into the condition that for any distinct $(x_1,...,x_s),(y_1,...,y_s)\in\Omega$, there exist $s+k-1\ge1$ elements 
$x_{s+1},...,x_{2s+k-1}\in\mathcal X$ with $x_{s+k}=y_1,...,x_{2s+k-1}=y_s$ and $\ell_{x_1,...,x_{s+1}}>0, ...,\ell_{x_{s+k-1},...,x_{2s+k-1}}>0$. 
In the setting of a dynamical system (Section~\ref{s:1.2}), Corollary~\ref{c:stat-inv-PDF} is in line with the well-known result that if $T$ is ergodic with respect to Lebesgue measure and $f$ specifies a $T$-invariant distribution, then $f$ specifies the unique $T$-invariant distribution (in fact, in all specific examples of $T$ given in this paper, $T$ is ergodic with respect to Lebesgue measure, see e.g.\ \cite{Karma,Fritz}). 

\be In continuation of Example~\ref{ex:2} when $\beta$ is the pseudo golden mean of order $s+1$, it follows by Theorem~\ref{t:MC} that $R^{(s)}\,|\,S$ is a Markov chain. Since the transition probabilities are given by \eqref{e:Pex}, the chain is seen to be
 irreducible. Thus, by Corollary~\ref{c:stat-inv-PDF}, $f_{\mathrm{inv}}$ specifies the unique shift invariant PDF or equivalently 
$T$-invariant  probability measure where $T(x)=\beta x-\lfloor\beta x\rfloor$. In fact, for any value of $\beta>1$,
 Gelfond \cite{gelfond} and Parry \cite{parry} have independently obtained the $T$-invariant distribution such that the corresponding PDF can be derived in closed form. 
\ee

\subsection{Coupling construction between $(X,N)$ and $(M_1,M_2)$}\label{s:Markov2}

Some notation is needed for the following theorem. 
For every $z\in\Omega$ and $i\in\mathbb N$, let $Y_i^{(s)}(z)=(Y_i(z),...,Y_{i+s-1}(z))$  and suppose $Y^{(s)}(z)=\{Y_i^{(s)}(z)\}_{i=1,2,...}$ is a Markov chain on $\Omega$ with transition matrix $P$ and $Y_1^{(s)}(z)$ following the PMF given by row $z$ in $P$. 
For $i\in\mathbb N$, denote 
$P_i^{(s)}(z)$ the distribution of $Y_i^{(s)}(z)$, and
 $P^{i}$ the $i$-step transition matrix of $Y^{(s)}(z)$, i.e., $P^{i}$ is the product of $P$ with itself $i$-times. Let $P^{i}_{z,w}$ be entry $(z,w)$ of $P^{i}$ and $\|\cdot\|_{\mathrm{TV}}$ the total variation distance for signed finite measures on $\Omega$.
Recall that uniform geometric ergodicity of a Markov chain with transition matrix $P$ means that 
there is a unique invariant PMF $\pi_{\mathrm{inv}}$ with respect to $P$ 
and for some constants $\alpha$ and $\beta<1$ and for all $(z,i)\in\Omega\times\mathbb N$,
\begin{equation}\label{e:unifergodicity}
\|P_i^{(s)}(z)-\Pi_{\mathrm{inv}}\|_{\mathrm{TV}}=\frac12\sum_{w\in\Omega}|P^{i}_{z,w}-\pi_{\mathrm{inv}}(w)|\le \alpha\beta^i
\end{equation}
where $\Pi_{\mathrm{inv}}$ is the  probability measure on $\Omega$ with PMF $\pi_{\mathrm{inv}}$.
 If $\Omega$ is finite, then uniform geometric ergodicity is equivalent to irreducibility and aperiodicity of the chain. When $\Omega$ is countable, see
Section 16.2.1 in \cite{Meyn} for conditions of
uniform geometric ergodicity. These criteria of uniform geometric ergodicity translate into conditions on the lengths of the intervals at level $s+1$. This was illustrated in the remark after the proof of Corollary~\ref{c:stat-inv-PDF} when discussing irreducibility. Aperiodicity holds if e.g.\ $\ell_{x_1,...,x_{2s}}>0$ and $\ell_{x_1,...,x_{2s+1}}>0$ for all $(x_1,...,x_{2s+1})\in\mathcal X^{2s+1}$.

\begin{theorem}\label{t:3}
Assume \eqref{e:LSC}, \eqref{e:MC2}, and
\eqref{e:unifergodicity}, and let $\pi_{\mathrm{inv}}$ be the unique $P$-invariant PMF and $f_{\mathrm{inv}}$  the corresponding shift invariant PDF. 
 Then there is a time $t\in\mathbb N$, a positive probability $\epsilon_t>0$ (given by \eqref{e:epst} below), and a coupling of $(X,N)$ with 
 two independent geometrically distributed random variables $M_1$ and $M_2$ with the same mean $1/\epsilon_t$ and support $\mathbb N$ such that the following items (A)--(C) hold. Let $b=t+s-1$, $M=b(M_1+M_2-1)$, and $K=\max\{N,s\}+M-s$. Then: 
\begin{enumerate}
\item[(A)] $(X_1,...,X_{\max\{N,s\}},N)$ is independent of both $X^{[K]}\sim f_{\mathrm{inv}}$ and  $X^{(s)}_{>K}$ which is a stationary Markov chain with initial PMF $\pi_{\mathrm{inv}}$ and transition matrix $P$.
\item[(B)] For any $n\in\mathbb N_0$, conditioned on $K\le n$, one has $X^{[n]}\sim f_{\mathrm{inv}}$. 
\item[(C)] $(X_1,...,X_{\max\{N,s\}},N)\perp\! \! \!\perp (M_1,M_2)$.
\end{enumerate} 
\end{theorem}


Theorem~\ref{t:3} and the following corollary are proved and exemplified in Section~\ref{s:proof}. 

\begin{corollary}\label{c:coupling2}
Assume 
\eqref{e:LSC}, \eqref{e:MC2}, and
\eqref{e:unifergodicity}, let $f_{\mathrm{inv}}$ and $K$ be as in Theorem~\ref{t:3}, and let
$\mu_{\mathrm{inv}}$ be the probability measure with PDF $f_{\mathrm{inv}}$. Then there is a coupling between $(X,K)$ and a random variable $Z\sim\mu_{\mathrm{inv}}$ such that $X^{[n]}=Z^{[n]}$ when $K+s\le n$, and so
\begin{equation}\label{e:slut}
\|Q^{[n]}-\mu_{\mathrm{inv}}\|_{\mathrm{TV}}\le \mathrm P(K>n)\to0\quad\mbox{as }n\to\infty.
\end{equation}
\end{corollary}

Theorem~\ref{t:3} and Corollary~\ref{c:coupling2} are discussed together with Corollaries~\ref{c:Luroth} and \ref{c:coupling1} in Section~\ref{s:f2}.

\subsection{Read once algorithm}\label{s:read}

The proof in Section~\ref{s:proof} of Theorem~\ref{t:3} is based on the read once algorithm developed by Wilson  \cite{wilson}. The algorithm is stated below for the setting used in the present paper.

The Markov chain $Y^{(s)}(z)$ can be specified as a stochastic recursive sequence (SRS) so that
\begin{equation}\label{e:srs11}
Y_{i}(z)=g(Y_{i-s}(z),...,Y_{i-1}(z),V_{i})
,\quad i=1,2,...,
\end{equation}
where $V_1,V_2,...$ are IID random variables, $(Y_{1-s}(z),...,Y_0(z))=z$,
and there are many possible choices of the updating function $g$ and for the distribution of $V_1$: Without loss of generality, assume $\mathcal X\subset\mathbb R$ is countable, and for specificity, let $V_1\sim\mathrm{Unif}(0,1)$, and for $(x_1,...,x_s)\in\Omega$, let 
\begin{equation}\label{e:inversion}
(0,1)\ni v\mapsto g(x_1,...,x_s,v)=F^-(v\,|\,x_1,...,x_s)
\end{equation} 
be the generalized inverse of the CDF
$$F(x\,|\,x_1,...,x_s)=\sum_{k:\,k\le x,\,(x_2,...,x_s,k)\in\Omega} p_{x_1,...,,x_{s},k},\quad x\in\mathbb R,$$
where $(x_2,...,x_s,k)$ is interpreted as $k$ if $s=1$. 
Note that all the chains $Y^{(s)}(z)$ with $z\in\Omega$ are coupled via $(V_1,V_2,...)$. 

For an integer $t\ge s$, letting $b=t+s-1$, then
\[B=B(V_1,...,V_b)=\{(Y_1(z),...,Y_{b}(z))\,|\,z\in\Omega\}\]
is called a block of length $b$.
The block agrees with $\{(Y^{(s)}_1(z),...,Y^{(s)}_t(z))\,|\,z\in\Omega\}$ and the notation 
$B(V_1,...,V_b)$ is used to stress that the block depends only on $V_1,...,V_b$. Let $O(z,V_1,...,V_b)=Y^{(s)}_t(z)$, which is interpreted as the output of $B$ when $z$ is the input in $B$. Call $B$  a coalescent block with common output $O(V_1,...,V_b)=O(z,V_1,...,V_b)$ if this is the same for all $z\in\Omega$.
 Let
\begin{equation}\label{e:epst}
\epsilon_t=\mathrm P(B\mbox{ is a coalescent block}),
\end{equation}
which is an increasing function of $t$.   
Suppose $\epsilon_t>0$, and consider the sequence of blocks 
\[B_i=B_i(V_{(i-1)b+1},...,V_{ib}),\quad i=1,2,...,\]
which are IID copies of $B$. Now,
Wilson's read once algorithm consists of three steps:
\begin{enumerate}
\item[(A)] Draw $B_1,...,B_{M_1},B_{M_1+1},...,B_{M_1+M_2}$ where $B_{M_1}$ and $B_{M_2}$ are the first and second coalescent blocks. Let $O_{M_1b}=O(V_{(M_1-1)b+1},...,V_{M_1b})$ be the common output of $B_{M_1}$.
\item[(B)] If $M_2\ge2$, for $i=0,...,M_2-2$, use the output $O_{(M_1+i)b}$ of $B_{M_1+i}$ as the input in the following block, which gives the output $$O_{(M_1+i+1)b}=O(O_{(M_1+i)b},V_{(M_1+i)b+1},...,V_{(M_1+i)b+b})$$
in $B_{M_1+i+1}$.
\item[(C)] Return $O=O_M$ where $M=(M_1+M_2-1)b$. 
\end{enumerate}

It follows from \cite{Foss,wilson} that the existence of a $t\in\mathbb N$ such that $\epsilon_t>0$ is equivalent to \eqref{e:unifergodicity}, and it has already been noticed that \eqref{e:unifergodicity} implies the unique existence of the invariant PMF $\pi_{\mathrm{inv}}$. Thus, 
the following lemma follows from 
\cite{wilson}. 

\begin{lemma}\label{l:3OLD}
Assume \eqref{e:MC2} and
\eqref{e:unifergodicity}, and let $\pi_{\mathrm{inv}}$ be the unique $P$-invariant PMF.
 Then 
 $O\sim\pi_{\mathrm{inv}}$.
 \end{lemma}

\subsection{Proofs 
and example}\label{s:proof}

\begin{proof} 
To prove Theorem~\ref{t:3}, assume \eqref{e:unifergodicity} and let $\pi_{\mathrm{inv}}$ be the unique $P$-invariant PMF and $f_{\mathrm{inv}}$  the corresponding shift invariant PDF. Using Theorem~\ref{thm1}, $X=0.X_1X_2...\sim f$ and $N$ can be generated as follows.
First, generate $(X_1,...,X_N)$ from the PMF \eqref{e:PXN}. If $N<s$ then generate $(X_{N+1},...,X_s)$ such that $(X_{N+1},...,X_s)$ conditioned on $(X_1,...,X_N)=(x_1,...,x_n)$ has PMF 
$$p(x_{n+1},...,x_s\,|\,x_1,...,x_n)=\ell_{x_1,...,x_s}/\ell_{x_1,...,x_n}\quad \mbox{whenever $\ell_{x_1,...,x_n}>0$.}$$ 
This is in accordance with \eqref{e:fdd}. Second, independent of $(X_1,...,X_{\max\{N,s\}},N)$, generate IID $V_1,V_2,...\sim\mathrm{Unif}(0,1)$ and use the SRS \eqref{e:srs11} when generating
$$X_{>{\max\{N,s\}}}= 
Y^{(s)}(X_{\max\{N,s\}-s+1},...,X_{\max\{N,s\}}).$$ 
Thus $(X_1,...,X_{\max\{N,s\}},N)$ is independent of the output $O$ of the read once algorithm. 
 By Lemma~\ref{l:3OLD}, 
$$X^{(s)}_{\max\{N,s\}+M-s+1}=\left(X_{\max\{N,s\}+M-s+1},...,X_{\max\{N,s\}+M}\right)=O\sim\pi_{\mathrm{inv}}.$$ 
Combining these facts with Corollaries~\ref{c:inv-PDF} and \ref{c:stat-inv-PDF}, 
(A) in Theorem~\ref{t:3} follows. From (A) follows (B). Finally, since $(X_1,...,X_{\max\{N,s\}},N)\perp\! \! \!\perp (V_1,V_2,...)$, (C) follows from the read once algorithm. 
Hence
the proof of Theorem~\ref{t:3} is completed.
\end{proof}

\begin{proof} To prove  Corollary~\ref{c:coupling2},
imagining that the first block starts just after time $\max\{N,s\}$. Then $\max\{N,s\}+M$ is the time at  
which the output of the read once algorithm has been produced. 
Hence, using (A) in Theorem~\ref{t:3}, along similar lines as in the proof of Corollary~\ref{c:coupling1},
 Corollary~\ref{c:coupling2} is obtained.
\end{proof}

\be\label{ex:final}
Consider again Example~\ref{ex:2} when $\beta$ is the pseudo golden mean of order $s+1$. Since $\Omega$ is finite and a Markov chain with transition matrix $P$ given by \eqref{e:Pex} is irreducible and aperiodic, $\epsilon_t>0$ for all sufficiently large $t$. So for large $t$, the assumptions in Theorem~\ref{t:3} 
are satisfied, and hence the results in Theorem~\ref{t:3} and Corollary~\ref{c:coupling2} apply.

Consider the case $s=1$. 
Solving 
\eqref{e:inv} to obtain $\pi_{\mathrm{inv}}$ and then using 
\eqref{e:finv} together with the fact that $\beta^2=\beta+1$, one gets
$$\pi_{\mathrm{inv}}=\left(\frac{\beta^2}{1+\beta^2},\frac{1}{1+\beta^2}\right),\quad f_{\mathrm{inv}}(x)=
\begin{cases}
(1+2\beta)/(2+\beta)&\text{if }0<x<1/\beta\\
(1+\beta)/(2+\beta)&\text{if }1/\beta<x<1.
\end{cases}$$ 
Choosing the updating function $g$ as in \eqref{e:inversion} and
letting $V\sim\mathrm{Unif}(0,1)$, then   
	\[\epsilon_1=P(F^-(V\,|\,0)=F^-(V\,|\,1))=P(V\le\beta^{-1})=\beta^{-1}=0.6180...\]
	This is for the case $t=1$, so the block length is just $b=1$
and the expected running time of the read once algorithm is just
$$\mathrm E (M_1+M_2)=2/\epsilon_1=\sqrt 5+1=3.2361...$$
So it may be expected that $R_n\sim f_{\mathrm{inv}}$ for small values of $n$. 
\ee

\section{Concluding remarks}\label{s:conclusion}

This section concludes with a discussion of the main results of this paper, including some open problems for future research. 

Assuming only almost everywhere lower semi-continuity of $f$, 
 Theorem~\ref{thm1} establishes the general results of this paper. In my opinion,  the two most important general results are that $S\perp\! \! \!\perp U\sim\mathrm{Unif}(0,1)$ and $R\,|\,S$ follows a distribution which does not depend on $f$ but it is specified by ratios of the lengths of intervals in the nested subdivisions, cf.\ \eqref{e:fdd}. 
 The first result leads to a discussion in Section~\ref{s:f1} of
statistical modelling, and the importance of the latter result is discussed in Section~\ref{s:f2}. 

\subsection{Statistical models}\label{s:f1}

After Theorem~\ref{thm1}  it was noticed that $S=(X_1,...,X_N)$ is a sufficient statistic when $N$ is treated as a latent variable. Hence, for likelihood based statistical inference, no matter if it is in a classical/frequentist or Bayesian setting,    
the likelihood function based on $(X,N)$ is equivalent to the likelihood based on the discrete random variable
$S$, cf.\ \cite{B-N}. So a natural question is how to specify 
 the distribution of $S$. 
 
 Statistical models where $N$ is expected to be small would possibly suffice in many applications:
For the case of base-$q$ expansions, specific examples of PDFs $f$ in \cite{HerbstEtAl3} show that $N$ will be small when $q$ is not small and $f$ does not deviate much from $f_{\mathrm{inv}}=1$. 

Since $S$ is a (discrete) random variable of varying dimension, one modelling approach could be inspired by the construction of point process models \cite{Daley,JM}. For example, in analogy with the use of Janossy densities when specifying finite point process densities,   by specifying the distribution of first $N$ and next $S\,|\,N$. Or, in the same spirit as when Gibbs distribution models for finite point process are defined, specifying the distribution of $S$ with respect to some natural reference measure under which e.g.\  $X_1,...X_N\,|\,N$ are IID and independent of $N$. 

In connection to Bayesian inference, another question is how to construct a random PDF $f(\cdot\,|\,Y)$, assuming $Y$ is a stochastic process  
such that 
$\mathrm P(X\in B\,|\,Y)=\int_B f(x\,|\,Y)\,\mathrm dx$
for any $B\in\mathcal B$ (the class of Borel sets included in $I$).
One possibility is the mixture of finite Polya tree distributions given by \eqref{e:mixture} below which is obtained as follows.
 Let $I$ be the unit interval and
  $\mathcal X=\{0,1\}$. 
Denote $B(\alpha_0,\alpha_1)$ the beta-distribution with shape parameters $\alpha_0\ge 0$ and $\alpha_1\ge 1$ such that $\alpha_0+\alpha_1>0$, where if $k\in\{0,1\}$ and $\alpha_k=0$ then $B(\alpha_0,\alpha_1)$ is concentrated at the point $k-1$ and hence
  the value of $\alpha_{1-k}>0$ plays no role.  
  Suppose
$Y=\{Y_{x_1,...,x_n}\,|\,(x_1,...,x_n)\in\{0,1\}^n,\,n\in\mathbb N\}$
where 
the $Y_{x_1,...,x_{n-1},0}$'s (interpreted as $Y_0$ if $n=0$) are independent $B(\alpha_{x_1,...,x_{n-1},0},\alpha_{x_1,...,x_{n-1},1})$-distributed,
$Y_{x_1,...,x_{n-1},1}=1-Y_{x_1,...,x_{n-1},0}$, and
  $\alpha_{x_1,...,x_{n}}=0$ if $\ell_{x_1,...,x_{n-1}}>0$ and $\ell_{x_1,...,x_{n}}=0$. 
Let $N$ be a non-negative integer-valued random variable which is independent of $Y$.
Conditioned on $(Y,N)$, 
for $n=0,...,N$, let the $X_n$'s be independent and let $Y_{X_1,...,X_{n-1},0}$ be the probability that $X_{n}=0$. Thereby, conditioned on $Y$, $S=(X_1,...,X_N)$ has been generated. In accordance with Theorem~\ref{thm1},
  conditioned on $S$,  independent of $Y$,
if $N=0$ then generate 
$X\sim\mathrm{Unif}(I)$ 
else generate $E\sim\mathrm{Unif}(0,\ell_{X_1,...,X_N})$ and set $X=a_{X_1,...,X_N}+E$. Thus, conditioned on $Y$ and $N=n$ with $\mathrm P(N=n)>0$, $X$ follows a random PDF $f(\cdot\,|\,Y,N=n)$ which for $x\in I_{x_1,...,x_n}$ is given by
\[f(x\,|\,Y,N=n)=\prod_{j=1}^n Y_{x_1,...,x_{j-1},0}^{1-x_j}Y_{x_1,...,x_{j-1},1}^{x_j}
\] 
where $\prod_{j=1}^0\cdots$ is interpreted as 1.
This random PDF, where $N=n$ is fixed, 
specifies
a finite Polya tree distribution in the sense of Lavine \cite{lavine2}; if $N$ was replaced by $\infty$ when generating $X_1,X_2,...$ above, the random probability distribution given by 
$\mathrm P(0.X_1X_2...\in B\,|\,Y)$ with $B\in\mathcal B$ would be a Polya tree distribution as defined by Lavine \cite{lavine}. 
 Arguments for considering instead a finite Polya tree distribution can be found in \cite{lavine2,WongMa}. For instance, a Polya tree distribution (in contrast to the finite Polya tree distribution) is not always absolutely continuous (unless the $\alpha$'s increase sufficiently rapidly which leads to problems of robustness, cf.\ \cite{lavine2}), and if it is absolutely continuous then the random PDF is almost surely discontinuous almost everywhere. Note that $X\,|\,Y$ specifies a random probability distribution $\mathcal P$ on $\mathcal B$ which is absolutely continuous with random PDF
\begin{equation}\label{e:mixture}
f(x\,|\,Y)=\sum_{n=0}^\infty f(x\,|\,Y,N=n)\mathrm P(N=n)
\end{equation}
(setting $f(x\,|\,Y,N=n)=0$ if $\mathrm P(N=n)=0$).
The use of $\mathcal P$ for statistical modelling and a comparison with other ways of imposing random stopping rules in a Polya tree distribution (including allowing optional stopping as in \cite{WongMa}) {\color{black} will be investigated in joint work with Mario Beraha \cite{mario}.}


Yet another interesting question is how to model random nested partitions? For instance, if $I$ is a unit interval and $\mathcal X=\{0,1\}$ is binary, one may consider a distribution of the lengths $\{\ell_{x_1,...,x_n}\,|\,(x_1,...,x_n)\in\{0,1\},n\in\mathbb N\}$ which is of a similar type as for $Y$ above and is independent of $(N,Y)$, and then proceed in a similar way as above when specifying a model for $S$ conditioned on both $Y$ and the random nested partitions {\color{black}(this will also be investigated in \cite{mario}).}
 How should this be extended to multi-variate cases (where the ideas behind randomized partitioning schemes in \cite{WongMa} may be a source of inspiration)?
How should one create random
 models for the intervals used to define the Markovian constructions in Section~\ref{s:2.1}, including the case of
  generalized L\"uroth series in Example~\ref{ex:a}?
These questions are left for future research.

\subsection{How important is the Markovian structure?}\label{s:f2}

 Birkhoff's ergodic theorem is a major tool in metric number theory when studying ergodic dynamical systems in connection to stationary digit systems $X_1,X_2,...$, see e.g.\ \cite{Karma,Dajani2,Fritz}. Birkhoff's theorem is in close agreement with Corollaries~\ref{c:Luroth} and \ref{c:coupling1}, Theorem~\ref{t:3}, and Corollary~\ref{c:coupling2}, although there are differences. To see this, suppose $X_{>0}$ is stationary and ergodic. Here, ergodicity means that $\mathrm P(X\in A)\in\{0,1\}$ whenever $A\subseteq I$ is a Borel set such that $.x_1x_2...\in A$ implies $.x_2x_3...\in A$. Let $h:I\mapsto\mathbb R$ be a Borel function where the mean $\mathrm Eh(X)$ exists. Then Birkhoff's theorem states that ``the time average converges to the space average'':
 $$\frac1m\sum_{n=1}^m h(X^{[n]})\to\mathrm Eh(X)\quad \mbox{as }m\to\infty.$$ 
 On the other hand, the results in the present paper state that under certain regularity conditions, but without assuming $X_{>0}$ is stationary, it is 
 not possible to distinguish between the distributions of the remainders $X^{[n]},X^{[n+1]},...$
  when $n$ is larger than a certain random time: At the random time $N+1$ (under the conditions given in Corollary~\ref{c:Luroth}) or $\max\{N,s\}+M$ (under the conditions given in Theorem~\ref{t:3}) and further on in time, 
 the remainders are exactly following the shift-invariant PDF $f_{\mathrm{inv}}$. 
 Moreover, under the conditions given in Corollaries~\ref{c:coupling1} and \ref{c:coupling2}, the distribution of $X^{[n]}$ converges towards the invariant distribution with respect to the total variation distance. 
 
For this convergence result the assumption that $f$ is 
LSC almost everywhere may be replaced by other assumptions: For a general PDF $f$, if there is only one PDF invariant under the shift, namely $f_{\mathrm{inv}}$, and if $\int_I|g^{[n]} - f_{\mathrm{inv}}|\to 0$ for a PDF $g^{[n]}$ corresponding to a smooth PDF $g$ (using a similar notation as for $f^{[n]}$ and $f$ in Section~\ref{s:coupling-const2}), the triangle inequality and the inequality $\int_I|f^{[n]} - g^{[n]}|\le \int_I|f - g|$ give the convergence result. Moreover, bounds for the rate of convergence  are given in \cite{HerbstEtAl} for the case of base-$q$ expansions (Example~\ref{ex:a}), in \cite{HerbstEtAl2} for $\beta$-expansions when $\beta$ is the golden mean (Example~\ref{ex:2}), and in \cite{arXiv} when $\beta$ is the pseudo golden mean (Example~\ref{ex:2}). For these cases the assumption that $f$ is LSC almost everywhere is indeed not needed. 

For cases where the Markovian structure \eqref{e:MC2} is not satisfied for any $s\in\mathbb N_0$, it would of course also be interesting to establish similar convergence results, to study the distribution of $R\,|\,S$ given by \eqref{e:fdd} (which do not depend on $f$), and if possible
to establish a coupling between $X$ and a random time $K$ (depending on $f$) so that $X^{[K]}\sim f_{\mathrm{inv}}$. Here, one particular interesting case is continued fraction representations, cf.\ Example~\ref{ex:contfrac}, where $f_{\mathrm{inv}}=f_{\mathrm G}$  is given by \eqref{e:fG}.

\subsubsection*{Acknowledgements} 
Supported by The Danish Council for Independent Research —
Natural Sciences, grant DFF – 10.46540/2032-00005B.
I am grateful to Anne Marie Svane and Ira Herbst for a careful reading of a draft of this paper and their useful comments. Discussions with Charlene Kalle and Simon Kristensen on various details, including continued fractions, are appreciated. I thank Mario Beraha for pointing my attention to Polya tree distributions.







\begin{thebibliography}{99}

%

\bibitem{B-N} 
	{\sc Barndorff-Nielsen, O.E.} (1978). {\em Information and Exponential Families in Statistical Theory.} Wiley, Chichester.
		
\bibitem{Barrionuevo}
	{\sc Barrionuevo, J., Burton, R.M., Dajani, K. and Kraiikamp, C.}\ (1996). Ergodic properties of generalized {L}\"{u}roth series. {\em Acta Arith.} {\bf 74}, 311--327.
	
\bibitem{mario} {\sc Beraha, M. and M{\o}ller, J.}\ (2025). Generalized finite Polya tree distributions for
statistical modelling. In preparation.
		
\bibitem{part1}
	{\sc Cornean, H., Herbst, I., M{\o}ller, J., St{\o}ttrup, B.B. and Studsgaard, K.S.}\ (2022). Characterization of random variables with stationary digits. {\em J.\ Appl.\ Prob.} {\bf 59}, 931--947.
	
\bibitem{part2}
	{\sc Cornean, H., Herbst, I., M{\o}ller, J., St{\o}ttrup, B.B. and Studsgaard, K.S.}\ (2023). Singular distribution functions for random variables with stationary digits.  {\em Methodol. Comput. Appl. Probab.} {\bf 25}, article no. 31. Available at arXiv:2201.01521.

\bibitem{Karma}
	{\sc Dajani, K. and Kalle, C.}\ (2021).
	{\em A First Course in Ergodic Theory.} Chapman and Hall/CRC.
	
\bibitem{Dajani2}
	{\sc Dajani, K. and Kraaikamp, C.}\ (2002).
	{\em Ergodic Theory of Numbers.} Mathematical Association of America.
	
\bibitem{Daley}
	{\sc Daley, D.J. and Vere-Jones, D.}\ (2003).
	{\em An Introduction to the Theory of Point Processes. Volume I: Elementary Theory and methods.} Springer-Verlag, New York.
	
\bibitem{Foss}
	{\sc Foss, S.G. and Tweedie, R.L.}\ (1998). Perfect simulation and backward coupling. {\em Stoch. Models} {\bf 14}, 187--203.
	
\bibitem{gauss}
	{\sc Gauss, C.F.}\ (1976). 
	{\em Mathematisches Tagebuch 1796--1814.} Akademische Verlagsgesellschaft, Geest \& Portig K.G., Leipzig.
	
\bibitem{gelfond}
	{\sc Gelfond, A.O.}\ (1959). A common property of number systems. {\em Izv. Akad. Nauk. SSSR. Ser. Mat.} {\bf 23}, 809--814.
	
\bibitem{HerbstEtAl}
	{\sc Herbst, I.~W., Møller, J. and Svane, A.~M.}\ (2024). How many digits are needed? {\em Methodol. Comput. Appl. Probab.}, {\bf 26}, article no.\ 5. 
	
\bibitem{HerbstEtAl3}
	{\sc Herbst, I.W., Møller, J. and Svane, A.M.}\ (2023). How many digits are needed? Extended version of \cite{HerbstEtAl} available at arXiv:2307.06685. 
		
\bibitem{HerbstEtAl2}
	{\sc Herbst, I.W., Møller, J. and Svane, A.M.}\ (2023).  
	The asymptotic distribution of the remainder in a certain base-$\beta$ expansion.
	Submitted for journal publication.
	Available at arXiv:2312.09652. 
	
\bibitem{arXiv}
	{\sc Herbst, I.W., Møller, J. and Svane, A.M.}\ (2024). 
	The asymptotic distribution of the scaled remainder for pseudo golden ratio expansions of a continuous random variable.
	Available at arXiv:2404.08387. 
	
\bibitem{Kin}
	{\sc  Khintchine, A.Y.}\ (1963).
	{\em Continued Fractions.} Noordhoff, Groningen. 

\bibitem{lavine}
	{\sc Lavine, M.}\ (1992). Some aspects of {Polya} tree distributions for statistical modelling. {\em Ann. Stat.}, {\bf 20}, 1222--1235. 

\bibitem{lavine2}
	{\sc Lavine, M.}\ (1994). More aspects of {Polya} tree distributions for statistical modelling. {\em Ann. Stat.}, {\bf 22}, 1161--1176. 

\bibitem{Luroth}
{\sc L\"{u}roth, J.}\ (1883). Ueber eine eindeutige {E}ntwickelung von {Z}ahlen in eine unendliche {R}eihe. {\em Math. Annalen}  {\bf 21}, 411--423.

\bibitem{Meyn}
	{\sc  Meyn, S.P. and Tweedie, R.L.}\ (1993).
	{\em Markov Chains and Stochastic Stability.} Springer Verlag, London.

\bibitem{JM}
	{\sc M{\o}ller, J. and Waagepetersen, R.P.}\ (2004).
	{\em Statistical Inference and Simulation for Spatial Point Processes.} Chapman \& Hall/CRC, Boca Raton.

\bibitem{parry}
	{\sc Parry, W.}\ (1960). On the $\beta$-expansions of real numbers. {\em Acta Math. Acad. Sci. Hungar.} {\bf 11}, 401--416.
	
\bibitem{renyi}
	{\sc R\'{e}nyi, A.}\ (1957). Representations for real numbers and their ergodic properties. {\em Acta Math. Acad. Sci. Hungar.} {\bf 8}, 477--493.

\bibitem{Fritz}
	{\sc Schweiger, F.}\ (1995). {\em Ergodic Theory of Fibered Systems and Metric Number Theory.} Clarendon Press, Oxford.
	

\bibitem{thorisson}
	{\sc Thorisson, H.}\ (2000) {\em  Coupling, Stationarity, and Regeneration.} Springer, New York.
	
\bibitem{wilson}
	{\sc Wilson, D.B.}\ (2000). How to couple from the past using a read-once source of randomness. {\em Random Struct. Alg.} {\bf 16}, 85--113. 
	
\bibitem{WongMa}
	{\textsc Wong, W.H. and Ma, L.}\ (2010). Optional {Pólya} tree and {Bayesian} inference. {\em Ann. Stat.} {\bf 38}, 1433--1459. 
	
\end{thebibliography}
\end{document}